\documentclass[twoside,leqno,11pt]{article}
\usepackage{revstat}
\usepackage[utf8]{inputenc}
\usepackage{natbib}
\usepackage{bm}
\usepackage{dsfont}

\newcommand{\E}{\operatorname{E}}
\newcommand{\expec}{\E}
\newcommand{\EE}{\mathds{E}}
\newcommand{\RR}{\mathds{R}}
\newcommand{\reals}{\RR}
\newcommand{\dto}{\xrightarrow{d}}
\newcommand{\vto}{\xrightarrow{v}}
\newcommand{\wto}{\xrightarrow{w}}
\newcommand{\diff}{\mathrm{d}}
\newcommand{\1}{I}

\begin{document}
\title{MAX-STABLE MODELS FOR MULTIVARIATE EXTREMES}
\renewcommand{\titleheading}
             {Max-stable models}  
\author{\authoraddress{Johan Segers}
                      {ISBA,
                       Universit\'e catholique de Louvain,\\
                       Belgium
                       \ (johan.segers@uclouvain.be)}
}
\renewcommand{\authorheading}
             {Johan Segers}  

\maketitle

\begin{abstract}
Multivariate extreme-value analysis is concerned with the extremes in a multivariate random sample, that is, points of which at least some components have exceptionally large values. Mathematical theory suggests the use of max-stable models for univariate and multivariate extremes. A comprehensive account is given of the various ways in which max-stable models are described. Furthermore, a construction device is proposed for generating parametric families of max-stable distributions. Although the device is not new, its role as a model generator seems not yet to have been fully exploited.
\end{abstract}


\begin{ams}
60G70, 62G32
\end{ams}

\mainpaper  

\section{INTRODUCTION}  
\label{Sec1}              

Multivariate extreme-value analysis is concerned with the extremes in a multivariate random sample, that is, points of which at least some components have exceptionally large values. Isolating a single component brings us back to univariate extreme-value theory. In this paper, the focus will rather be on the dependence between extremes in different components. The issue of temporal dependence will be ignored, so that the dependence will be understood as cross-sectional only.

Mathematical theory suggests the use of max-stable models for univariate and multivariate extremes. The univariate margins must be one of the classical extreme-value distributions, Fr\'echet, Gumbel, and extreme-value Weibull, unified in the generalized extreme-value distributions. For the dependence structure, however, matters are more complicated. A complete characterization in the multivariate case was given in \cite{dehaan:resnick:1977}, describing extremal dependence in terms of spectral measures on a subset of the unit sphere. Statistically, this formulation is not always the most practical one, and a large number of other concepts have been proposed.

The aim of this paper is to provide a comprehensive account of the various ways in which max-stable models are described (Section~\ref{sec:fundamentals}). Second, a construction device is proposed for generating parametric families of max-stable distributions (Section~\ref{sec:models}). The device is not new as it appears already for instance in the theory of regularly varying multivariate time series in \cite{basrak:segers:2009} or in the concept of a $D$-norm in \cite{falk:huesler:reiss:2010}. Still, its role as a model generator seems not yet to have been fully appreciated.

Inference on multivariate extremes via max-stable models for joint tails is an exciting field that is still in development. Inference methods can be nonparametric or parametric, and in the latter case, they can be likelihood-based, frequentist as well as Bayesian, or based on other techniques such as the method of moments or minimum distance estimation.

Max-stable models have the drawback that they are too coarse to describe tails of multivariate distributions with asymptotic independence sufficiently accurately. More refined models exist and they are the topic of an extensive literature, originating from the seminal paper of \cite{ledford:tawn:1996}.

\section{FUNDAMENTALS}

\label{sec:fundamentals}

Max-stable distributions arise out of the study of the limit behaviour of vectors of componentwise maxima. Their dependence structure can be described via a number of dependence functions or via certain measures. All these objects are related and they highlight different aspects of the distribution of multivariate extremes.

\subsection{Convergence of vectors of componentwise maxima}


Let $\bm X_i = (X_{i1}, \ldots, X_{id})$, for $i = 1, \ldots, n$, be independent and identically distributed random vectors in dimension $d$. For each component $j = 1, \ldots, d$, consider the sample maximum, and collect these maxima in a new random vector:
\begin{align*}
  \bm M_n &= (M_{n1}, \ldots, M_{nd}), &
  M_{nj} &= \max(X_{1j}, \ldots, X_{nj}).
\end{align*}  
Observe that most of the time, the point $\bm M_n$ does not belong to the sample, as the maxima in different components will typically occur at different time instances. Still, the analysis of the large-sample distribution of $\bm M_n$ is a natural starting point for multivariate extreme-value theory.

Weak convergence of a sequence of random vectors implies weak convergence of each of the components. As in univariate extreme-value theory, it is therefore reasonable to apply increasing affine transformations to each of the margins and consider the sequence of random vectors
\begin{equation}
\label{eq:maxima}
  \biggl( \frac{M_{nj} - b_{nj}}{a_{nj}} : j = 1, \ldots, d \biggr), \qquad n = 1, 2, \ldots
\end{equation}
in terms of normalizing constants $a_{nj} > 0$ and $b_{nj}$. For each component $j = 1, \ldots, d$, the weak limit of $(M_{nj}- b_{nj}) / a_{nj}$ as $n \to \infty$ must be a univariate max-stable distribution, and necessary and sufficient conditions on the marginal distribution $F_j$ of the $j$th component $X_{ij}$ for such convergence to take place are well known, see for instance the monograph by \cite{dehaan:ferreira:2006}.

However, weak convergence of each of the $d$ components in \eqref{eq:maxima} is strictly weaker than joint convergence of the vector of normalized maxima. What is needed in addition is a condition on the dependence structure of the common joint distribution $F$ of the random vectors $\bm X_i$. A convenient way to describe this dependence is via the copula $C_1$ of $F$, that is,
\begin{equation}
\label{eq:Sklar}
  \Pr[ \bm X_i \le \bm x ] = F ( \bm x ) = C_1 \bigl( F_1(x_1), \ldots, F_d(x_d) \bigr).
\end{equation}
Assuming the margins $F_1, \ldots, F_d$ are continuous, as we will do henceforth, the copula $C_1$ of the distribution function $F$ in \eqref{eq:Sklar} is unique and can be obtained as the joint distribution function of the random vector $(F_1(X_{i1}), \ldots, F_d(X_{id}))$.

Weak convergence of a sequence of multivariate distribution functions to a limit with continuous margins is equivalent to weak convergence of the sequences of margins and of the sequence of copulas \citep{deheuvels:1984,galambos:1987}. The copula of the vector of component-wise maxima $\bm M_n$, and hence of any vector that is obtained by applying increasing transformations to each of its components, is given by
\begin{equation}
\label{eq:Cn}
  C_n(\bm u) = \bigl\{ C_1\bigl(u_1^{1/n}, \ldots, u_d^{1/n}\bigr) \bigr\}^n.
\end{equation}
This can be checked from the fact that the joint distribution function of $\bm M_n$ is $F^n$ while its margins are $F_j^n$ for $j = 1, \ldots, d$. Hence, in order for the normalized maxima in \eqref{eq:maxima} to converge in distribution to a nondegenerate limit, besides marginal convergence, the sequence of copulas $C_n$ must converge as well.

The copulas that can arise as weak limits of $C_n$ as $n \to \infty$ are called extreme-value copulas, that is, a copula $C$ is called an extreme-value copula if there exists a copula $C_1$ such that, as $n \to \infty$,
\begin{equation}
\label{eq:EVC}
  \lim_{n \to \infty} \bigl\{ C_1\bigl(u_1^{1/n}, \ldots, u_d^{1/n} \bigr) \bigr\}^n = C(u_1, \ldots, u_d).
\end{equation}
Extreme-value copulas arise as the class of possible limit copulas of vectors $\bm M_n$ as $n \to \infty$. The copula $C_1$ is said to be in the domain of attraction of $C$. An extensive survey of the literature on extreme-value copulas is given in \cite{gudendorf:segers:2010}.

The class of extreme-value copulas coincides with that of max-stable copulas, defined as follows. A copula $C$ is max-stable if, for all $\bm u \in [0, 1]^d$ and $k = 1, 2, \ldots$,
\[
  C(\bm u) = \bigl\{ C \bigl( u_1^{1/k}, \ldots, u_d^{1/k} \bigr) \bigr\}^k.
\]
In the setting of componentwise maxima of independent random samples, the previous identity means that the copula $C_k$ of the random vector $\bm M_k$ is the same for every sample size $k$. Clearly, a max-stable copula is also an extreme-value copula, being in its own domain of attraction. Conversely, each extreme-value copula can be shown to be max-stable: in \eqref{eq:EVC}, partition the sample of size $n = mk$ in $m$ blocks of size $k$ and let $m$ tend to infinity for fixed $k$. Since the limit must not depend on $k$, the max-stability relation follows.

In summary, we have found that nondegenerate limit distributions of vectors of appropriately normalized componentwise maxima have extreme-value margins and an extreme-value or max-stable copula. Specifically, if
\[
  \Pr \biggl[ \bigcap_{j=1}^d \biggl\{ \frac{M_{nj} - b_{nj}}{a_{nj}} \le x_j \biggr\} \biggr]
  \xrightarrow{w} G(x_1, \ldots, x_d), \qquad n \to \infty,
\]
then necessarily
\[
  G(x_1, \ldots, x_d) = C \bigl( G_1(x_1), \ldots, G_d(x_d) \bigr)
\]
with extreme-value margins $G_1, \ldots, G_d$ and an extreme-value copula $C$. Convergence of margins and convergence of copulas being two isolated issues, we can ignore the former and rather focus on the latter. In fact, the way in which the components are normalized is immaterial, as long as the transformations applied to the components are increasing.

\subsection{Dependence functions}

Take logarithms and apply a linear expansion to see that \eqref{eq:EVC} is equivalent to 
\begin{multline}
\label{eq:C2ell}
  \lim_{n \to \infty} n \{ 1 - C_1(1 - n^{-1} x_1, \ldots, 1 - n^{-1} x_d) \} \\
  = - \log C(e^{-x_1}, \ldots, e^{-x_d}) 
  = \ell(\bm x), \qquad \bm x \in [0, \infty)^d.
\end{multline}
The limit $\ell$ is called the stable tail dependence function of $C$, going back to \cite{huang:1992} and \cite{drees:huang:1998}. The variable $n$ tending to infinity along the positive integers can be replaced by a variable $t$ tending to infinity along the positive reals.

The best known example is the Gumbel--Hougaard copula, for which $\ell_\theta(\bm x) = (x_1^\theta + \cdots + x_d^\theta)^{1/\theta}$ in terms of a parameter $\theta \in [1, \infty]$ \citep{gumbel:1960,hougaard:1986}. The function $\ell_\theta$ happens to be the $\theta$-norm of the vector $\bm x$. The fact that $\ell_\theta$ is a norm is no coincidence: in a remarkable paper by \cite{molchanov:2008}, a characterization is given of all the norms that can give rise to stable tail dependence functions. In \cite{falk:huesler:reiss:2010}, $\ell(\bm x)$ is called the $D$-norm of $\bm x$, with $D$ referring to the Pickands dependence function, see below.

Let $\bm X = (X_1, \ldots, X_d)$ denote a generic random vector in the original sample. The expression on the left-hand side in \eqref{eq:C2ell} contains the rescaled probability
\begin{multline}
\label{eq:or}
  1 - C_1(1 - x_1/n, \ldots, 1 - x_d/n) \\
  = \Pr[ F_1(X_{1}) > 1 - x_1/n \quad \text{or} \quad \ldots \quad \text{or} \quad F_d(X_{d}) > 1 - x_d/n ].
\end{multline}
This probability concerns the event that at least one among the $d$ components $X_{1}, \ldots, X_{d}$ should exceed a high percentile of its own distribution. The copula domain-of-attraction condition \eqref{eq:EVC}, originally involving the vector of componentwise sample maxima, has been replaced by a condition on the upper tail of a single random vector. This is akin to the familiar peaks-over-threshold approach in univariate extreme-value theory.

The tail copula, $R$, of \cite{schmidt:stadtmueller:2006} arises if all $d$ components are required to exceed a large percentile simultaneously:
\begin{multline*}
  \lim_{n \to \infty} n \Pr[ F_1(X_1) > 1 - x_1/n \quad \text{and} \quad \ldots \quad \text{and} \quad F_d(X_d) > 1 - x_d/n ] \\
  = R(\bm x), \qquad \bm x \in [0, \infty)^d.
\end{multline*}
Clearly, the relation between the functions $\ell$ and $R$ is governed by the inclusion-exclusion formula. In higher dimensions, $\ell$ is somewhat more convenient than $R$, as setting some components $x_j$ in the definition of $\ell$ to zero allows one to retrieve the lower-dimensional margins of the extreme-value copula. This is not possible for the tail copula $R$, as setting even a single $x_j$ to zero immediately yields $R(\bm x) = 0$. The difference between the two functions $\ell$ and $R$ is depicted in Figure~\ref{fig:lR}.

\begin{figure}
\begin{center}
\begin{tabular}{cc}
\includegraphics[width=0.45\textwidth]{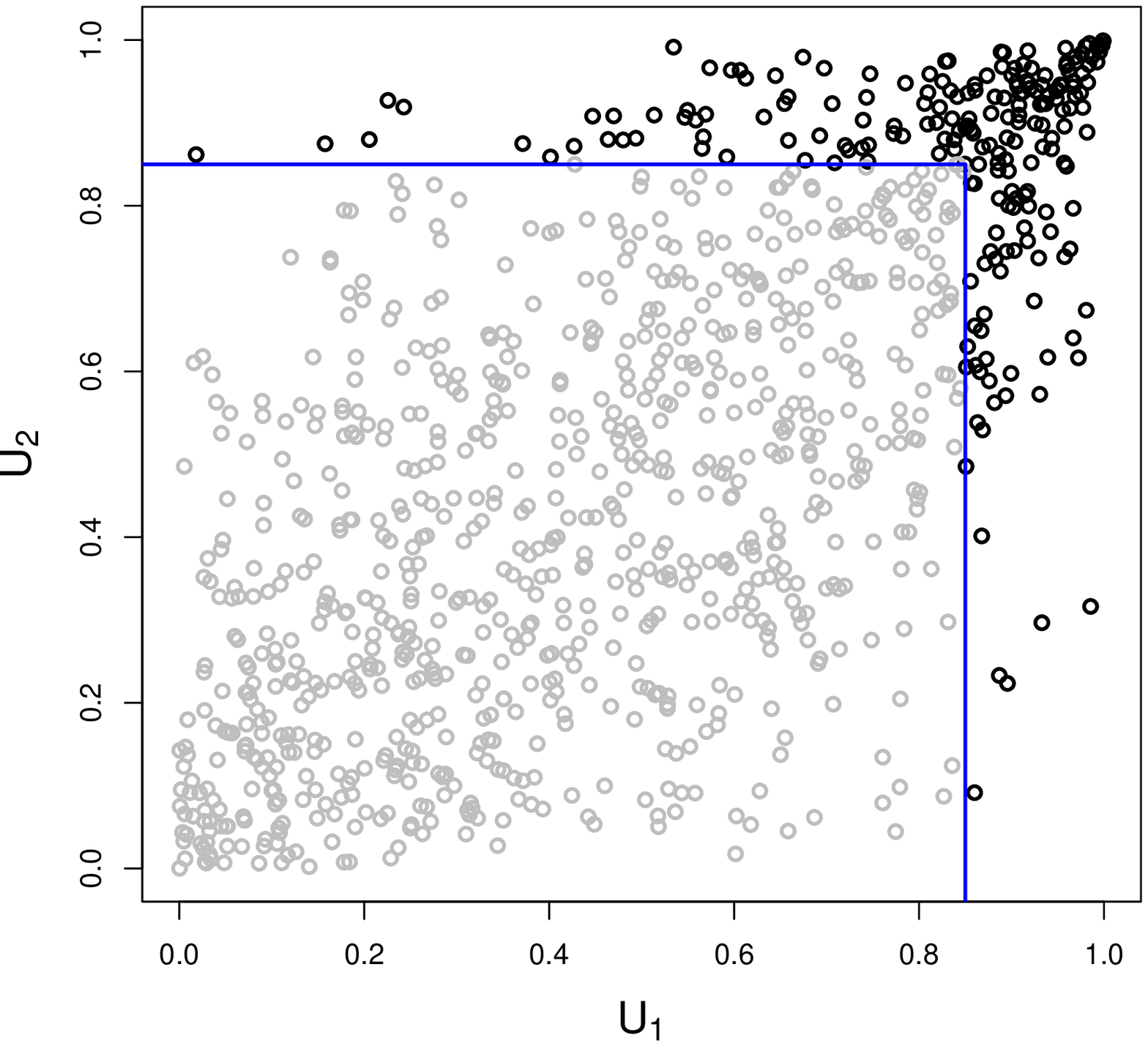}&
\includegraphics[width=0.45\textwidth]{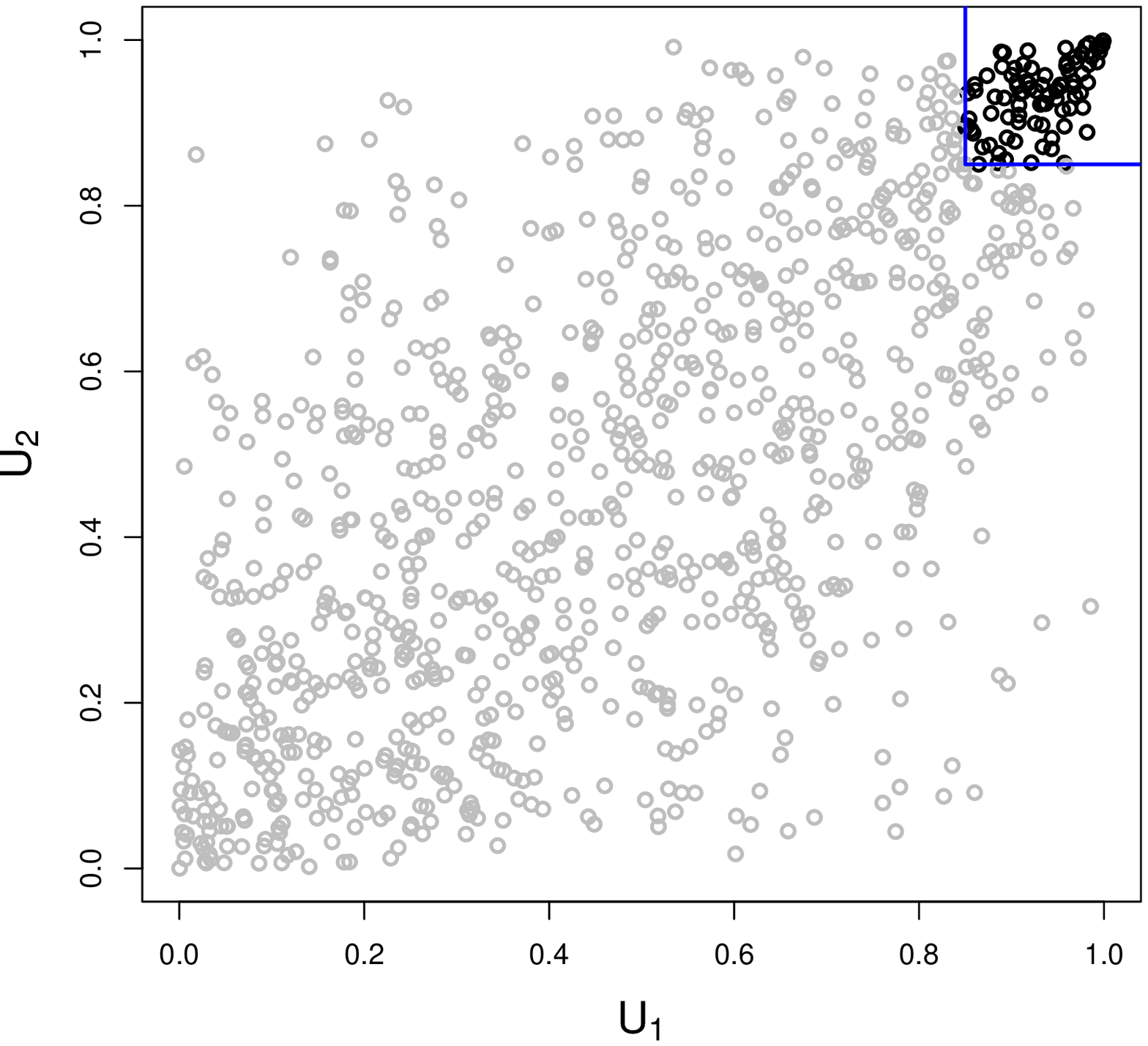}\\[1ex]
$U_1$ or $U_2$ are large&
$U_1$ and $U_2$ are large
\end{tabular}
\caption{Left: the stable tail dependence function $\ell$ -- at least one component should be large. Right: the tail copula $R$ -- all components must be large simultaneously. Here $U_j = F_j(X_j)$ for $j = 1, 2$. \label{fig:lR}}
\end{center}
\end{figure}

The dependence functions $\ell$ and $R$ are homogeneous:
\begin{align}
  \ell(a \, \bm x)
  &= \lim_{t \to \infty} t \, \bigl\{ 1 - C_1 \bigl( 1 - x_1/ (t/a), \ldots, 1 - x_d/ (t/a) \bigl) \bigr\} \nonumber \\
  &= \lim_{s \to \infty} a \, s \, \bigl\{ 1 - C_1(1 - x_1/s, \ldots, 1 - x_d/s) \bigr\} \nonumber \\
  &= a \, \ell(\bm x), \qquad a > 0, \; \bm x \in [0, \infty)^d, \label{eq:ellhomo}
\end{align}
and similarly for $R$. It is therefore sufficient to consider the restriction of these functions to the unit simplex $\Delta_{d-1} = \{ (w_1, \ldots, w_d) \in [0, 1]^d : w_1 + \cdots + w_d = 1 \}$. The restriction of $\ell$ to $\Delta_{d-1}$ is called the Pickands dependence function, $D$, after \cite{pickands:1981}. By homogeneity,
\begin{align*}
  \ell(\bm x) &= (x_1 + \cdots + x_d) \, D(w_1, \ldots, w_d), & w_j &= \frac{x_j}{x_1 + \cdots + x_d}.
\end{align*}
Frequently, the variable $w_d = 1 - w_1 - \cdots - w_{d-1}$ is suppressed from the notation and $D$ is written as a function of the arguments $(w_1, \ldots, w_{d-1})$ only.

The probability on the right-hand side of \eqref{eq:or} involves the union of the events $\{ F_j(X_j) > 1 - x_j/n \}$, each of which has probability $x_j / n$, provided $0 \le x_j \le n$. As a consequence, we have the elementary bounds
\begin{multline*}
  \max(x_1/n, \ldots, x_d/n) \\
  \le \Pr[ F_1(X_1) > 1 - x_1/n \quad \text{or} \quad \ldots \quad \text{or} \quad F_d(X_d) > 1 - x_d/n ] \\
  \le x_1/n + \cdots + x_d/n.
\end{multline*}
Multiplying by $n$ and letting $n$ tend to infinity, we obtain
\begin{equation}
\label{eq:bounds:ell}
  \max(x_1, \ldots, x_d) \le \ell(x_1, \ldots, x_d) \le x_1 + \cdots + x_d, \qquad \bm x \in [0, \infty)^d.
\end{equation}
By \eqref{eq:C2ell} and \eqref{eq:bounds:ell}, an extreme-value copula $C$ must satisfy 
\begin{equation}
\label{eq:bounds:C}
  u_1 \cdots u_d \le C(u_1, \ldots, u_d) \le \max(u_1, \ldots, u_n), \qquad \bm u \in [0, 1]^d.
\end{equation}
The lower and upper bounds in the two previous displays can be attained, corresponding to the extreme cases of independence and perfect association. In particular, max-stable models are positive quadrant dependent. In fact, in \cite{guillem:2000}, the stronger property is shown that bivariate extreme-value copulas are monotone regression dependent; see also Theorem~5.2.10 in \cite{resnick:1987}. 

By \eqref{eq:C2ell}, the copula $C$ can be given in terms of the tail dependence function through
\[
  C(u_1, \ldots, u_d) = \exp \{ - \ell ( - \log u_1, \ldots, - \log u_d ) \}, \qquad \bm u \in (0, 1]^d.
\]
In extreme-value theory, it is often convenient to standardize to other distributions than the uniform $(0, 1)$ law. The three most common forms are the unit Fréchet distribution, the Gumbel distribution, and the reverse exponential distribution, yielding respectively
\begin{align}
\label{eq:C:frechet}
  C(e^{-1/x_1}, \ldots, e^{-1/x_d}) &= \exp \{- \ell ( 1/x_1, \ldots, 1/x_d) \}, & \bm x &\in (0, \infty)^d, \\
\label{eq:C:gumbel}
  C(e^{-e^{-x_1}}, \ldots, e^{-e^{-x_d}}) &= \exp \{ - \ell (e^{-x_1}, \ldots, e^{-x_d}) \}, & \bm x &\in \reals^d, \\
\label{eq:C:weibull}
  C(e^{x_1}, \ldots, e^{x_d}) &= \exp \{ - \ell(-x_1, \ldots, -x_d) \}, & \bm x &\in (-\infty, 0)^d.
\end{align}
When using unit Fréchet margins, the notation $V(x_1, \ldots, x_d) = \ell(1/x_1, \ldots, 1/x_d)$ is often employed too.

\subsection{The intensity measure}

The transformation of the components $X_j$ to uniform $(0, 1)$ random variables via the probability integral transform $F_j(X_j)$ has the disadvantage that all the action regarding the upper extremes is compressed to the neighbourhood of $1$. Instead, for a univariate sequence $\xi_1, \xi_2, \ldots$ of independent and identically distributed random variables with common distribution function $F_\xi$, define the first exceedance time of the level $x$ by
\[
  T(x) = \inf \{ i = 1, 2, \ldots : \xi_i > x \}.
\]
If $F_\xi(x) < 1$, then $T(x)$ will be a geometric random variable with success probability equal to $1 - F_\xi(x)$. Its expectation,
\[
  \E[T(x)] = \frac{1}{1-F_\xi(x)},
\]
is called the return time of the level $x$.

Now let us apply this return time transformation to each of the $d$ components of the random vector $(X_1, \ldots, X_d)$. The return time of observation $X_j$ is $Y_j = 1 / \{ 1-F_j(X_j) \}$. The law of $Y_j$ is unit Pareto rather than uniform on $(0, 1)$, as $\Pr[Y_j  > y] = \Pr[F_j(X_j) > 1-1/y] = 1/y$ for $y \geq 1$. We find that values of $X_j$ corresponding to high percentiles of $F_j$ are mapped to large values of $Y_j$. As is evident from Figure~\ref{fig:returntime}, extremes are magnified.

\begin{figure}
\begin{center}
\begin{tabular}{cc}
\includegraphics[width=0.45\textwidth]{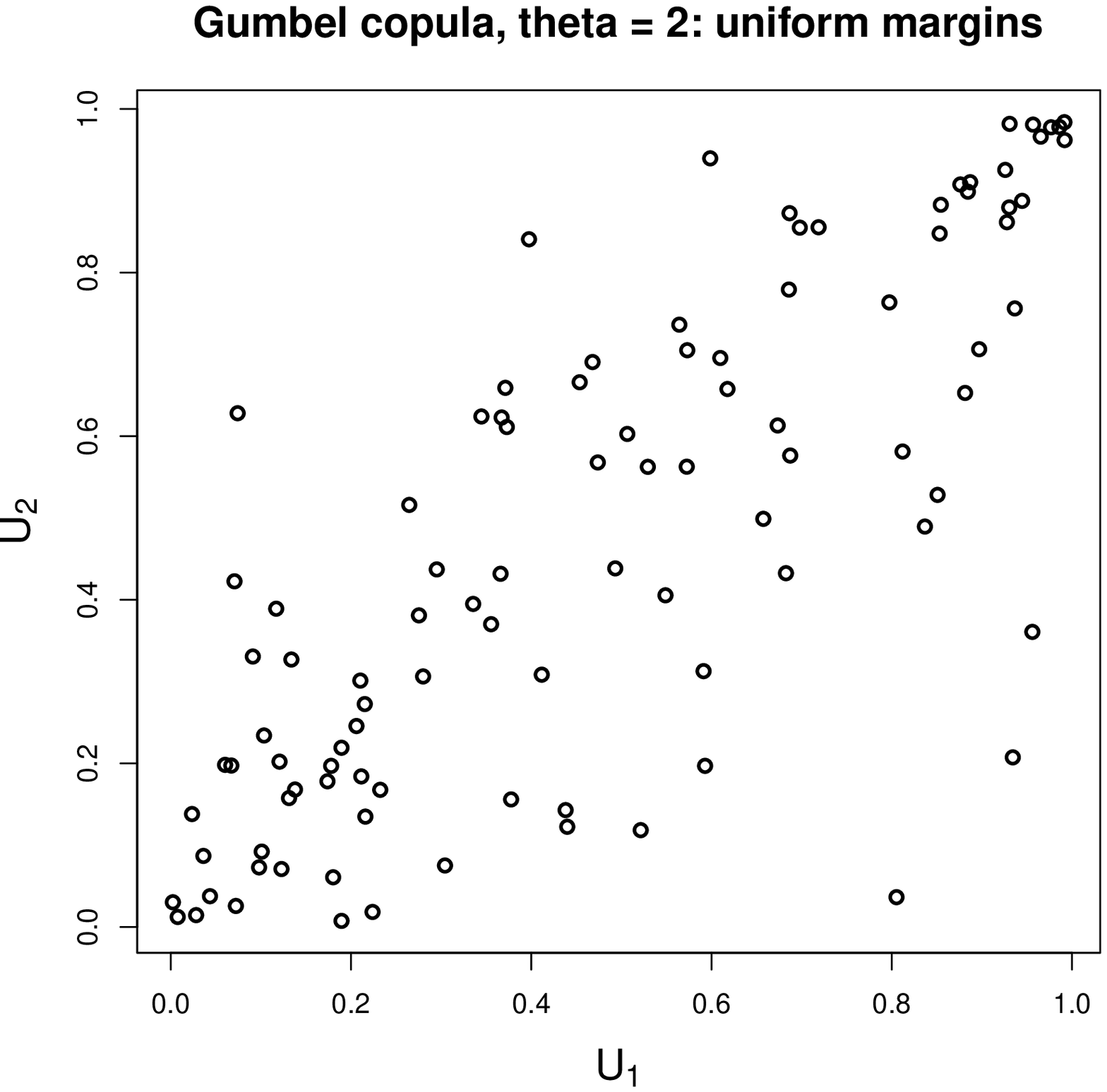}&
\includegraphics[width=0.45\textwidth]{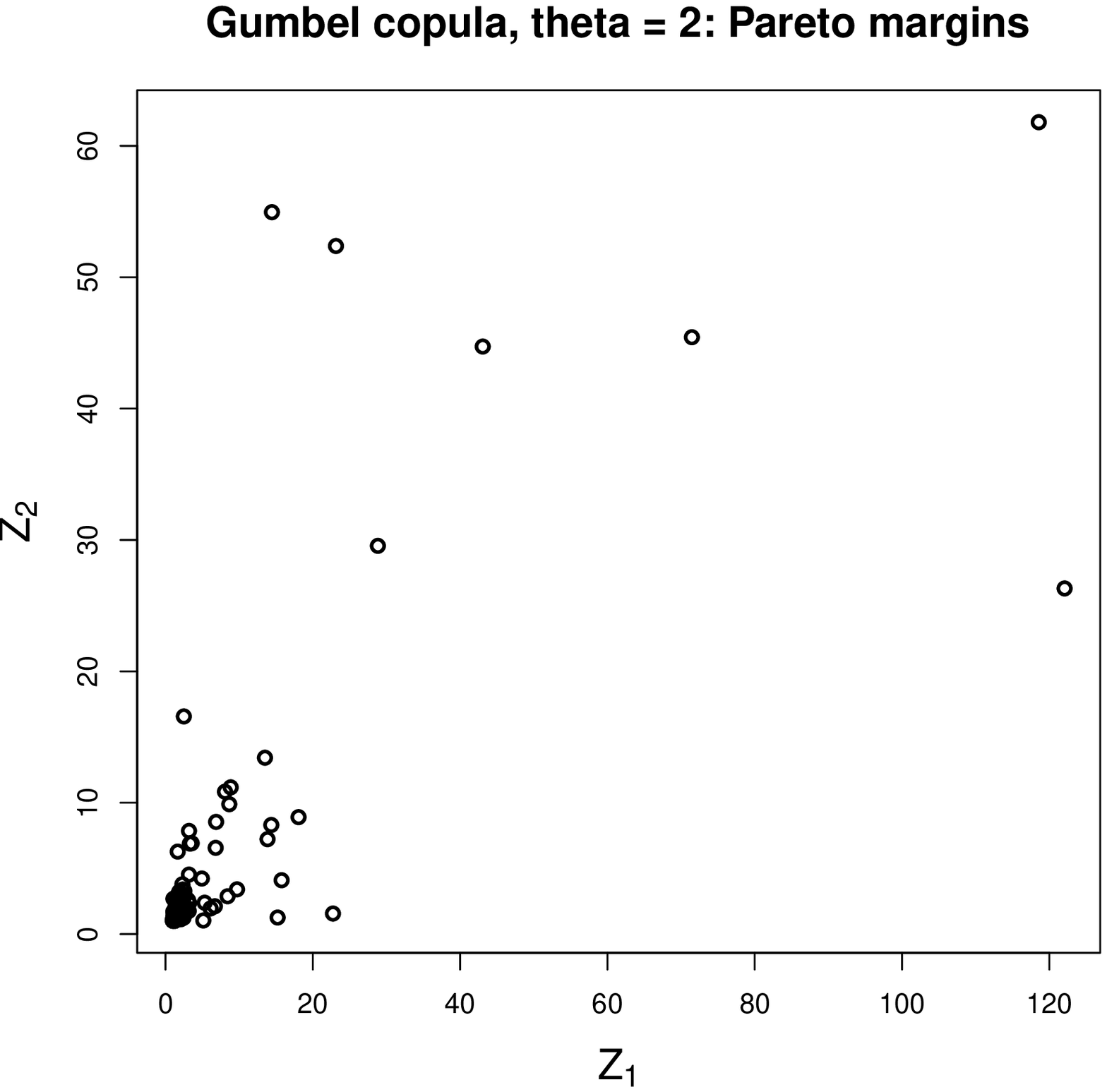}\\
$U_{ij} = F_j(X_{ij})$&
$Z_{ij} = 1 / \{ 1 - F_j(X_{ij}) \}$
\end{tabular}
\caption{Two views on a sample: uniform (left) versus Pareto (right)}
\label{fig:returntime}
\end{center}
\end{figure}


Now suppose that the copula $C_1$ is in the domain of attraction of an extreme-value copula with stable tail dependence function $\ell$. Equation~\eqref{eq:C2ell} says that the random vector $\bm Y = (Y_1, \ldots, Y_d)$ satisfies
\begin{align*}
  t \, \{ 1 - C_1(1 - x_1/t, \ldots, 1 - x_d/t) \}
  &= t \, \Pr \biggl[ \bigcup_{j=1}^d \{ Y_j > t / x_j \} \biggr] \\
  &= t \, \Pr \bigl[ \bm Y/t \in  \bigl( [0, \infty]^d \setminus [\bm 0, 1/\bm x] \bigr) \bigr] \\
  &\to \ell(\bm x), \qquad t \to \infty.
\end{align*}
It follows that on the space $\EE_d = [0, \infty]^d \setminus \{ \bm 0 \}$, there exists a measure $\mu$ such that
\begin{equation}
\label{eq:vague2mu}
  \expec \biggl[ \sum_{i=1}^n \1 (\bm Y_i/n \in \, \cdot \,) \biggr]
  = n \, \Pr[ \bm Y/n \in \, \cdot \, ] \vto \mu(\,\cdot\,), \qquad n \to \infty.
\end{equation}
The limit takes place in the mode of vague convergence of measures, meaning that $\lim_{n \to \infty} n \expec[ f(\bm Y / n) ] = \int_{\EE_d} f( \bm x ) \, \diff \mu(\bm x)$ for every bounded, continuous function $f$ on $\EE_d$ that vanishes in a neighbourhood of the origin. Intuitively, when $n$ grows large, the vector $\bm Y/n$ is pulled towards the origin, in the neighbourhood of which the function $f$ is zero. The intensity measure then only concerns the upper tail of the distribution of $\bm Y$. The first expression in the previous display shows that $\mu$ expresses the limiting average intensity in space of the normalized sample cloud $\{\bm Y_1 / n, \ldots, \bm Y_n / n\}$. The stable tail dependence function acts as a distribution function for the intensity measure $\mu$, as we have 
\[
  \ell(\bm x) = \mu \bigl( [0, \infty]^d \setminus [\bm 0, 1/\bm x] \bigr).
\]

\subsection{Extreme profiles}

Assume that the vector of return times $\bm Y = (Y_1, \ldots, Y_d)$ with $Y_j = 1 / \{ 1 - F_j(X_j) \}$ is large, that is, at least one of its components exceeds a high threshold. The relative sizes of the $d$ components then inform us about the extremal dependence: are some components large simultaneously or is one specific component dominating all the other ones? Specifically, for $\bm y \in [0, \infty)^d \setminus \{ 0 \}$ put
\begin{align*}
  r(\bm y) &= y_1 + \cdots + y_d \in (0, \infty), \\
  \bm w(\bm y) &= \bigl( y_j / r(\bm y) : j = 1, \ldots, d \bigr) \in \Delta_{d-1},
\end{align*}
to be thought of as the magnitude and the profile of the vector $\bm y$, respectively. The spectral measure $H$ lives on the unit simplex and is defined by
\[
  H(B) = \mu \bigl( \{ \bm y : r(\bm y) > 1, \, \bm w(\bm y) \in B \} \bigr),
\]
for Borel subsets $B$ of $\Delta_{d-1}$. 


The measure convergence in \eqref{eq:vague2mu} implies that
\begin{multline}
\label{eq:vague2H}
  \expec \biggl[ \sum_{i=1}^n \1 \{ r(\bm Y_i) > n, \, \bm w(\bm Y_i) \in \, \cdot \, \} \biggr] \\
  = n \, \Pr[ r(\bm Y) > n, \, \bm w(\bm Y) \in \, \cdot ]
  \wto H(\,\cdot\,), \qquad n \to \infty,
\end{multline}
with vague convergence being replaced by weak convergence because the state space $\Delta_{d-1}$ is compact. In particular,
\begin{equation}
\label{eq:weak2Q}
  \Pr[ \bm w(\bm Y) \in \, \cdot \mid r(\bm Y) > t ] \dto \frac{H(\,\cdot\,)}{H(\Delta_{d-1})} = Q(\,\cdot\,), \qquad t \to \infty,
\end{equation}
meaning that the conditional distribution of the profile $\bm w(\bm Y)$ given that the magnitude $r(\bm Y)$ is large converges in distribution to the normalized spectral measure $Q$.

The spectral measure $H$ and the profile distribution $Q(\,\cdot\,)$ are alternative, equivalent ways of describing the extreme-value copula $C$. Indeed, homogeneity of $\ell$ in \eqref{eq:ellhomo} implies homogeneity of $\mu$:
\begin{equation}
\label{eq:muhomo}
  \mu(a \, \cdot \,) = a^{-1} \, \mu(\, \cdot \,), \qquad a > 0.
\end{equation}
As a consequence, the intensity measure satisfies
\begin{align}
\label{eq:mu2H:slice}
  \mu \bigl( \{ \bm y : r(\bm y) > z, \, \bm w(\bm y) \in B \} \bigr)
  &= z^{-1} \, \mu \bigl( \{ \bm y \mid r > 1, \, \bm w \in B \} \bigr) \\
  &= z^{-1} \, H(B) \nonumber
\end{align}
for $z > 0$ and for Borel sets $B$ of $\Delta_{d-1}$. That is, when expressing a point $\bm y$ in the coordinates $(r, \bm w)$, the intensity measure $\mu$ factorizes into a product measure on $(0, \infty) \times \Delta_{d-1}$ given by $r^{-2} \diff r \, H(\diff \bm w)$. Equation~\eqref{eq:mu2H:slice} leads to
\begin{equation}
\label{eq:H2mu}
  \int f(\bm y) \, \mu(\diff \bm y) = \int_{\Delta_{d-1}} \int_0^\infty f(r \bm w) \, r^{-2} \, \diff r \, H(\diff \bm w)
\end{equation}
for $\mu$-integrable functions $f$, showing how to recover $\mu$ and thus $\ell$ and $C$ from $H$. The special case where $f$ is equal to the indicator function of the set $\{ \bm y : \max_j x_j y_j > 1 \}$ for some $\bm x \in [0, \infty)^d$ yields, after some computation,
\begin{equation}
\label{eq:H2ell}
  \ell(\bm x) = \int_{\Delta_{d-1}} \max_{j=1, \ldots, d} (w_j x_j) \, H(\diff \bm w).
\end{equation}
Incidentally, this representation of $\ell$ implies that $\ell$ must be convex. By specializing the bounds in \eqref{eq:bounds:ell} to the unit vectors in $\reals^d$, one finds that the spectral measure $H$ must satisfy the constraints
\begin{equation}
\label{eq:constraints}
  1 = \ell(\bm e_j) = \int_{\Delta_{d-1}} w_j \, H(\diff \bm w), \qquad j = 1, \ldots, d.
\end{equation}
It follows that the total mass of $H$ is equal to 
\[ 
  H(\Delta_{d-1}) = \int (w_1 + \cdots + w_d) \, H(\diff \bm w) = d. 
\]
Thanks to this property, it is possible to recover the spectral measure $H$ from the profile distribution $Q$. From \eqref{eq:constraints}, it then follows that a random vector $\bm W = (W_1, \ldots, W_d)$ on $\Delta_{d-1}$ with law equal to $Q$ must satisfy
\begin{equation}
\label{eq:constraints:Q}
  \expec_Q[ W_j ] = \int_{\Delta_{d-1}} w_j \, Q(\diff \bm w) = 1/d, \qquad j = 1, \ldots, d.
\end{equation}
In Section~\ref{sec:models}, we will see that any such law $Q$ can appear as the profile distribution of a $d$-variate max-stable distribution.

In case of asymptotic independence, $\ell(\bm x) = x_1 + \cdots + x_d$, the profile distribution $Q$ is equal to the discrete uniform distribution on the $d$ vertices of $\Delta_{d-1}$: asymptotically, only one component can be large at a time. In the case of asymptotic perfect dependence, $\ell(\bm x) = \max(x_1, \ldots, x_d)$, the profile distribution $Q$ is degenerate at the center $(1/d, \ldots, 1/d)$ of $\Delta_{d-1}$: all components are equally large. These two extreme cases are depicted in Figure~\ref{fig:indepcomonotone}.

\begin{figure}
\begin{center}
\begin{tabular}{cc}
\includegraphics[width=0.45\textwidth]{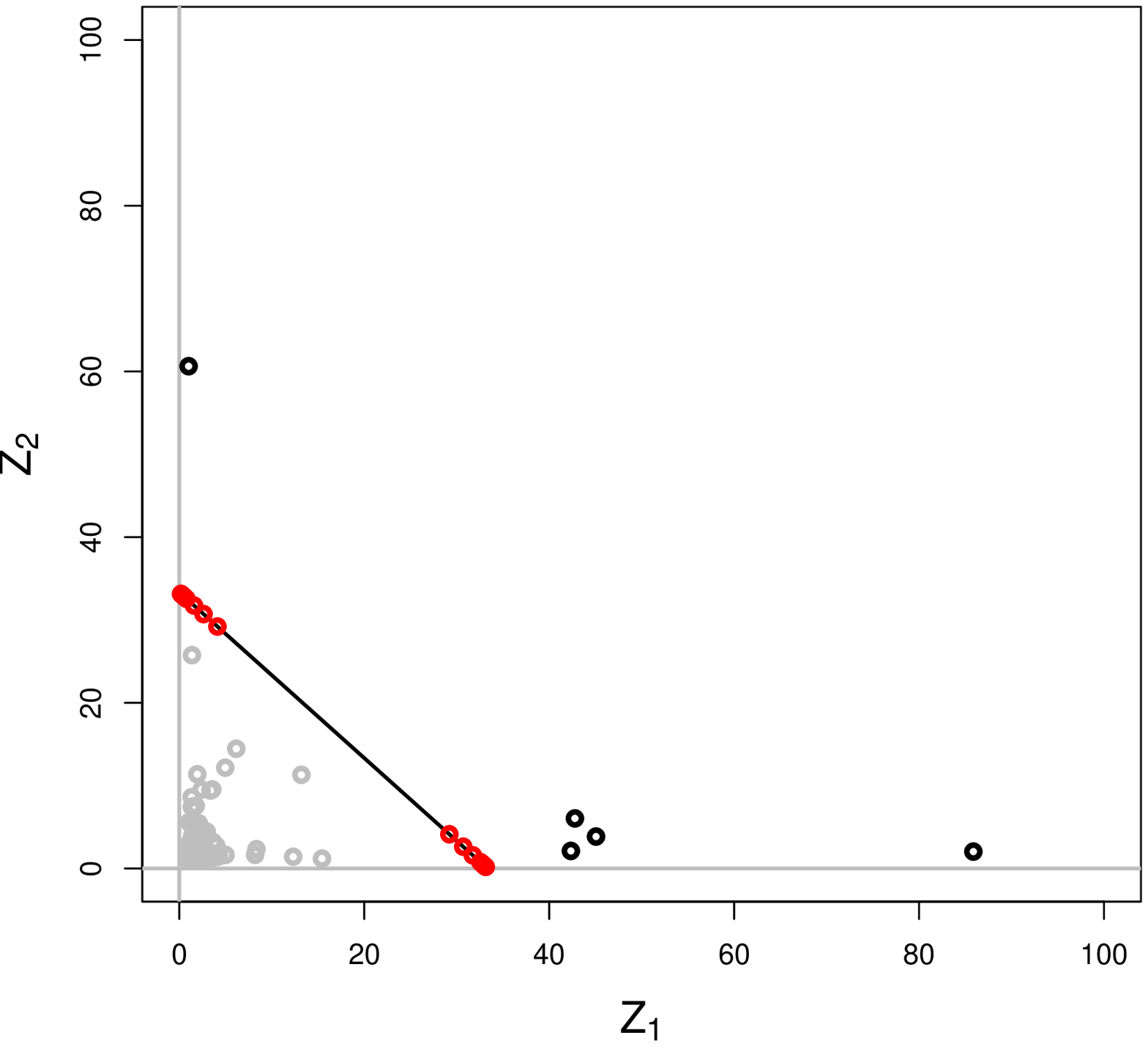}&
\includegraphics[width=0.45\textwidth]{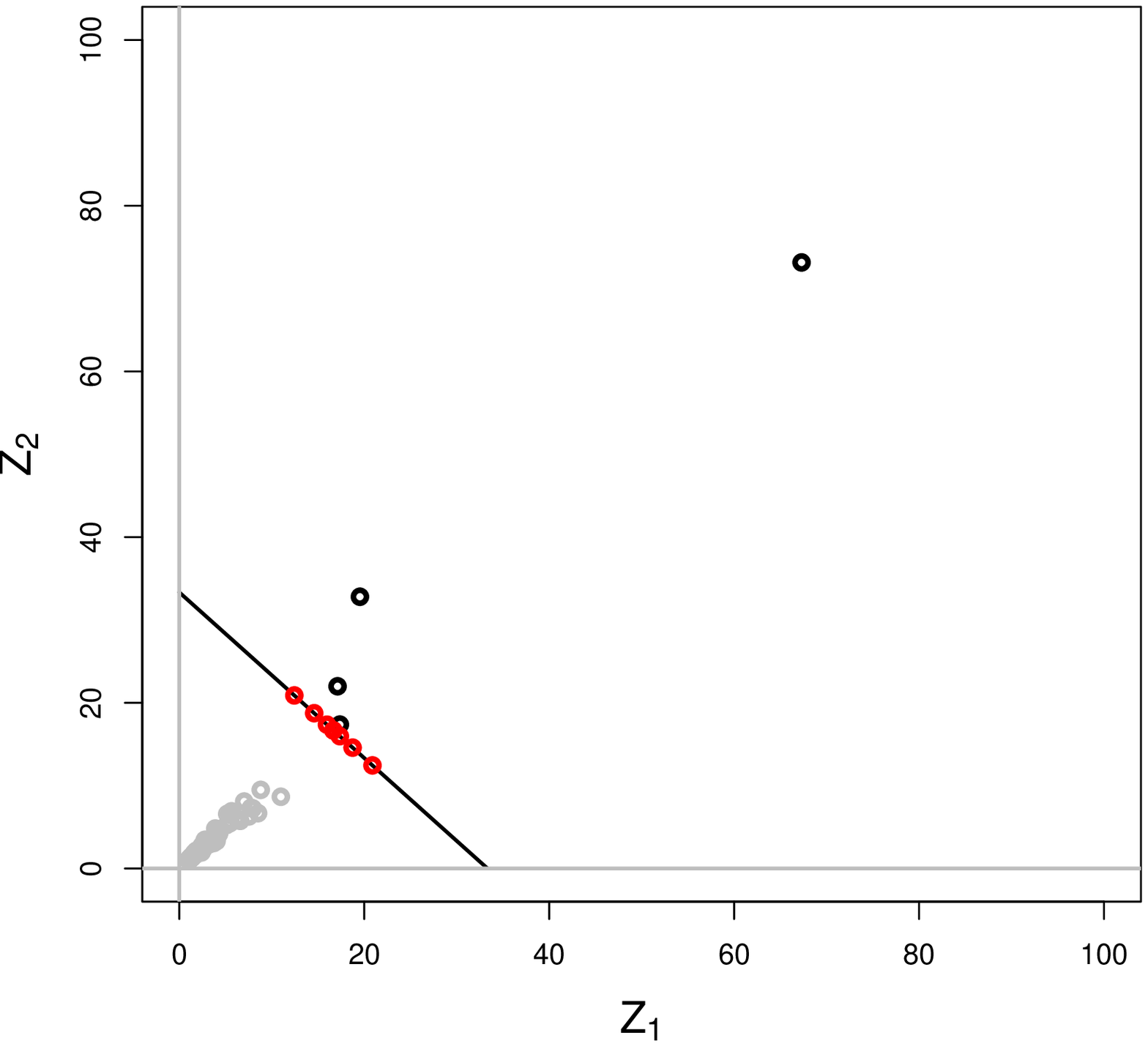}
\end{tabular}
\end{center}
\caption{Asymptotic independence (left) versus asymptotic perfect dependence (right).}
\label{fig:indepcomonotone}
\end{figure}

To show the ease with which coefficients related to extremal dependence can be computed, consider the random variable
\[
  N(t) = \sum_{j=1}^d \1 \{ F_j(X_j) > 1 - 1/t \},
\]
counting the number of components that exceed a high percentile. The following dependence coefficients have natural interpretations.
\begin{itemize}
\item \emph{Trouble is in the air:}
\[
  \lim_{t \to \infty} t \, \Pr[ N(t) \ge 1] = \ell(1, \ldots, 1) = \int_{\Delta_{d-1}} \max(w_1, \ldots, w_d) \, H(\diff \bm w).
\]
\item \emph{Multiple failures:}
\[
  \lim_{t \to \infty} t \, \Pr[ N(t) \ge k] = \int_{\Delta_{d-1}} w_{(d-k+1)} \, H(\diff \bm w),
\]
for $k = 1, \ldots, d$, where $w_{(1)} \le \cdots \le w_{(d)}$ denote the order statistics of $(w_1, \ldots, w_d)$.
\item \emph{The sky is falling:}
\[
  \lim_{t \to \infty} t \, \Pr[ N(t) = d ] = R(1, \ldots, 1) = \int_{\Delta_{d-1}} \min(w_1, \ldots, w_d) \, H(\diff \bm w).
\]
\item \emph{System collapse -- how bad will it get?}
\[
  \lim_{t \to \infty} \expec [ N(t) - k \mid N(t) \ge k ]
  = \frac{\int (w_{(1)} + \cdots + w_{(d-k)}) \, H(\diff \bm w)}{\int w_{(d-k+1)} \, H(\diff \bm w)},
\]
for $k = 1, \ldots, d-1$.
\end{itemize}

\section{CONSTRUCTING MODELS}
\label{sec:models}

There is a fairly large number of parametric max-stable models available; see for instance the overviews in \cite{kotz:nadarajah:2000} and \cite{bgst:2004}. In a search for flexible models in large dimensions, new families are still being constructed, as in \cite{ballani:schlather:2011}, \cite{boldi:davison:2007}, \cite{cooley:davis:naveau:2010} and \cite{fougeres:nolan:rootzen:2009}. In this section, a simple construction device will be proposed and illustrated.

From Section~\ref{sec:fundamentals}, we recall that max-stable models for extremal dependence can be represented either via the extreme-value copula $C$, the stable tail dependence function $\ell$, the intensity measure $\mu$, or the spectral measure $H$ and its normalized version, the profile distribution $Q$. However, as these objects must satisfy certain constraints, construction of parametric models is not obvious, particularly in high dimensions. Even if flexible parametric forms can be found, interpretation of the model parameters may not be obvious. In addition, when working with the spectral measure or profile distribution, the passage to lower-dimensional margins can be awkward, as the conditioning events in \eqref{eq:weak2Q} will be different according to which components are selected.

\subsection{A construction device}


Let $Z$ be a unit Fréchet random variable, that is, $\Pr(Z \le z) = \exp(-1/z)$ for $z > 0$. Let $\bm A = (A_1, \ldots, A_d)$ be a random vector, independent of $Z$, such that $0 < \expec [\max(A_j, 0)] < \infty$ for every $j \in \{1, \ldots, d\}$. Consider the random vector
\begin{equation}
\label{eq:model}
  \bm X = (X_1, \ldots, X_d) = (A_1 Z, \ldots, A_d Z).
\end{equation}
The max-stable attractor of the distribution of $\bm X$ can be explicitly calculated.

\begin{lemma}
For $x = (x_1, \ldots, x_d) \in (0, \infty)^d$, we have
\begin{multline}
\label{eq:recipe:G}
  \lim_{n \to \infty} \Pr[ X_1 \le nx_1, \ldots, X_d \le nx_d ]^n \\
  = \exp \{ - \expec [ \max(A_1/x_1, \ldots, A_d/x_d, 0) ] \}.
\end{multline}
\end{lemma}

\begin{proof}
Let $x \in (0, \infty)^d$. We have
\begin{align*}
  \Pr[ X_1 \le x_1, \ldots, X_d \le x_d ]
  &= \Pr[ A_1 Z \le x_1, \ldots, A_d Z \le x_d ] \\
  &= \Pr[ A_1/x_1 \le 1/Z, \ldots, A_d/x_d \le 1/Z ] \\
  &= \Pr[ 1/Z \ge \max(A_1/x_1, \ldots, A_d/x_d) ]
\end{align*}
The distribution of $1/Z$ is unit exponential. Since $A$ and $Z$ are independent,
\[
  \Pr[ X_1 \le x_1, \ldots, X_d \le x_d \mid A_1, \ldots, A_d]
  = \exp \{ - \max(A_1/x_1, \ldots, A_d/x_d, 0) \}.
\]
It follows that
\[
  \Pr[ X_1 \le x_1, \ldots, X_d \le x_d ]
  = \expec [ \exp \{ - \max(A_1/x_1, \ldots, A_d/x_d, 0) \} ].
\]

Let $A^{(1)}, A^{(2)}, \ldots$ be a sequence of independent and identically distributed copies of $A$. Fix positive integer $n$. We have
\begin{align*}
  \lefteqn{
  \Pr[ X_1 \le n x_1, \ldots, X_d \le n x_d ]^n
  } \\
  &= \bigl( \expec [ \exp \{ - \max(A_1/x_1, \ldots, A_d/x_d, 0) \} ] \bigr)^n \\
  &= \expec \biggl[ \exp \biggl\{ - \frac{1}{n} \sum_{i=1}^n \max(A_1^{(i)}/x_1, \ldots, A_d^{(i)}/x_d, 0) \biggr\} \biggr].
\end{align*}
Equation~\eqref{eq:recipe:G} now follows by the law of large numbers and the dominated convergence theorem.
\end{proof}

The margins of the limit distribution function, say $G$, in \eqref{eq:recipe:G} are equal to $G_j(x_j) = \exp \{ - \expec[ \max(A_j, 0) ] / x_j \}$ for $x_j > 0$. Assume that $\expec[ \max(A_j, 0) ] = 1$ for all $j = 1, \ldots, d$; this can always be achieved by rescaling the variables $A_j$ if necessary. In that case, the margins of $G$ are unit Fréchet. Comparing equations~\eqref{eq:recipe:G} and~\eqref{eq:C:frechet}, we find that the stable tail dependence function of $G$ is given by
\begin{equation}
\label{eq:recipe:ell}
  \ell_{\bm A}(x_1, \ldots, x_d) = \expec[ \max(x_1 A_1, \ldots, x_d A_d, 0) ]
\end{equation}


The spectral measure $H$ corresponding to $\ell$ in \eqref{eq:recipe:ell} can be identified too. Write $A_j^+ = \max(A_j, 0)$ and put $R = A_1^+ + \cdots + A_d^+$. On the event $R > 0$, define $W_j = A_j^+ / R$; on the event $R = 0$, the definition of $W_j$ is immaterial -- for definiteness, put $W_j = 1/d$ if $R = 0$. Note that $\bm W$ takes values in $\Delta_{d-1}$ and that $\expec[R] = \sum_{j=1}^d \expec[A_j^+] = d$. We have $A_j^+ = R W_j$ and thus
\begin{align*}
  \ell_{\bm A}(x_1, \ldots, x_d) 
  &= \expec [ R \, \max(W_1 x_1, \ldots, W_d x_d) ] \\
  &= \expec \bigl[ \expec[R \mid W_1, \ldots, W_d] \, \max(W_1 x_1, \ldots, W_d x_d) \bigr].
\end{align*}
Comparing this expression with \eqref{eq:H2ell}, we find that the spectral measure $H$ of $\ell_{\bm A}$ is given by
\begin{equation}
\label{eq:W2H}
  H(\diff \bm w) = \expec[R \mid \bm W = \bm w] \, \Pr[\bm W \in \diff \bm w],
\end{equation}
that is, $H$ is absolutely continuous with respect to the law of $\bm W$ with Radon-Nikodym derivative equal to $\expec[R \mid \bm W = \bm w]$. Similarly, the profile distribution $Q$ satisfies
\begin{equation}
\label{eq:W2Q}
  Q(\diff \bm w) = d^{-1} \expec[R \mid \bm W = \bm w] \, \Pr[\bm W \in \diff \bm w],
\end{equation}
Intuitively, this makes sense: profiles $\bm W$ that on average yield larger values of $R$ will have a larger contribution to the joint extremes of $\bm X$.

Incidentally, this construction shows that any probability distribution $Q$ on $\Delta_{d-1}$ satisfying \eqref{eq:constraints:Q} can appear as the profile distribution of a $d$-variate max-stable distribution. Indeed, let the random vector $\bm W$ on $\Delta_{d-1}$ have law $Q$ and put $A_j = d W_j$ for all $j = 1, \ldots, d$. As $A_1 + \cdots + A_d = d$ by construction, the law of the random vector $\bm X$ in \eqref{eq:model} is in the domain of attraction of a $d$-variate max-stable distribution with profile distribution equal to $Q$.

If the dimension $d$ is large, realistic models of extremal dependence should allow for the possibility that only some but not all components of a random vector are large simultaneously. In terms of the spectral measure or the profile distribution, this is encoded by the lower-dimensional faces of the unit simplex. For a non-empty subset $I$ of $\{1, \ldots, d\}$, let $\Delta_{d-1,I}$ denote the set of all $\bm w$ in $\Delta_{d-1}$ such that $w_j > 0$ if $j \in I$ and $w_j = 0$ otherwise. If the probability of the event $\{ \min_{j \in I} A_j > 0 \ge \max_{j \in I^c} A_j \}$ is non-zero, then by \eqref{eq:W2H}, the spectral measure and the profile distribution will put positive mass on $\Delta_{d-1, I}$. The set $I$ contains the indices of the components that are large.

\subsection{Examples}

The simplicity of \eqref{eq:recipe:ell} is appealing and the reader is invited to the apply the recipe in order to produce his or her own parametric models. In the remainder of the article, a number of well-known examples are worked out.

\begin{example}[Independence]
Suppose that $\Pr[ A_j > 0 \ge \max(A_i : i \neq j) ] = p_j$ for $p_j > 0$ and $p_1 + \cdots + p_d = 1$. Then only component can be large at the time. After standardization ensuring that $\expec[ \max(A_j, 0) ] = 1$ for all $j = 1, \ldots, d$, we find $\ell_{\bm A}(x_1, \ldots, x_d) = \expec[ \max(x_1 A_1, \ldots, x_d A_d, 0) ] = x_1 + \cdots + x_d$, the stable tail dependence function of the independence copula.
\end{example}

\begin{example}[Perfect dependence]
Suppose that $A_j = a_j B$ with probability one for all $j = 1, \ldots, d$, for some constants $a_j > 0$ and a random variable $B$ such that $\expec[ \max(B, 0) ]$ is positive and finite. Then the profile of an extreme vector is fixed. After standardization, $a_j = 1 / \expec[ \max(B, 0) ]$, the stable tail dependence function is that of perfect positive association, that is,  $\ell_{\bm A}(x_1, \ldots, x_d) = \max(x_1, \ldots, x_d)$.
\end{example}

\begin{example}[Discrete spectral measures]
\label{ex:discrete}
Suppose that the distribution of $\bm A$ is discrete with a finite number of atoms. Specifically, suppose that $\Pr[\bm A = \bm a_k] = p_k$ for $k \in \{1, \ldots, m\}$, where $\bm a_k \in \reals^d$ and $p_k \in (0, 1)$ such that $\sum_{k=1}^m p_k = 1$. Via standardization, ensure that $1 = \expec[ \max(A_j, 0) ] = \sum_{k=1}^m p_k a_{kj}^+$ for all $j = 1, \ldots, d$. Put $r_k = a_{k1}^+ + \cdots + a_{kd}^+$ and write $\bm a_k^+ = r_k \bm w_k$: if $r_k > 0$, then $w_{kj} = a_{kj}^+/r_k$, whereas if $r_k = 0$, then put $w_{kj} = 1/d$. It follows that
\begin{align*}
  \ell_{\bm A}(x_1, \ldots, x_d) 
  &= \sum_{k=1}^m p_k \max(a_{k1}^+ x_1, \ldots, a_{kd}^+ x_d) \\
  &= \sum_{k=1}^m (p_k r_k) \max(w_{k1} x_1, \ldots, w_{kd} x_d).
\end{align*}
We find that the spectral measure $H$ and the profile distribution $Q$ are discrete and are given by
\begin{align*}
  H &= \sum_{k=1}^m p_k r_k \delta_{\bm w_k}, &
  Q &= \sum_{k=1}^m d^{-1} p_k r_k \delta_{\bm w_k},
\end{align*}
with $\delta_{\bm w}$ a unit point mass at $\bm w$. The probabilities $p_k$ are tilted with the magnitudes $r_k$, giving higher prominence to profiles $\bm w_k$ that are associated to larger values of $r_k$.

Max-stable models with discrete spectral measures are called extreme-value factor models in \cite{einmahl:krajina:segers:2011}. Each of the $m$ possible outcomes $\bm a_k$ results in a different profile, according to the type of event or factor that triggered the extreme value.
\end{example}

\begin{example}[Random indicators]
Let 
\[
  \ell(x, y) = \expec[ \max( xA, yB , 0) ], \qquad (x, y) \in [0, \infty)^2,
\]
in terms of random variables $A$ and $B$ such that $\expec[ \max(A, 0) ] = \expec[ \max(B, 0) ] = 1$. Let $(I, J)$ be a pair of random indicators, independent of the pair $(A, B)$, such that $\Pr[ I = 1 ] = p$, $\Pr[ J = 1 ] = q$, and $\Pr[ I = J = 1] = r$. In the definition of $\ell$, replace the pair $(A, B)$ by the pair $(p^{-1} I A, q^{-1} J B)$; we assume that $0 < p \le 1$ and $0 < q \le 1$. The new stable tail dependence function is equal to
\begin{align*}
  \lefteqn{
  \expec[ \max( p^{-1} x I A, q^{-1} y J B, 0 ) ]
  } \\
  &= \expec[ \max( p^{-1} x A, q^{-1} y B ) ] \, \Pr[I = J = 1] \\
  &\quad \mbox{} + \expec[ \max( p^{-1} x A, 0 ) ] \, \Pr[I = 1, J = 0] + \expec[ \max( q^{-1} y B, 0 ) ] \, \Pr[I = 0, J = 1] \\
  &= \ell(p^{-1} x, q^{-1} y) \, r + p^{-1} x \, (p - r) + q^{-1} y \, (q - r)
\end{align*}
Writing $\alpha = r/p = \Pr[J = 1 \mid I = 1]$ and $\beta = r/q = \Pr[I = 1 \mid J = 1]$, we find
\begin{equation}
\label{eq:indicator:asymmetric}
  \ell_{\alpha, \beta}(x, y) = \ell(\alpha x, \beta y) + (1 - \alpha) x + (1 - \beta) y, \qquad (x, y) \in [0, \infty)^2.
\end{equation}
The new tail copula is simply
\begin{align*}
  R_{\alpha, \beta}(x, y) 
  &= x + y - \ell_{\alpha, \beta}(x, y) \\
  &= \alpha x + \beta y - \ell(\alpha x, \beta y)
  = R(\alpha x, \beta y), \qquad (x, y) \in [0, \infty)^2.
\end{align*}
This is an asymmetric, two-parameter extension of the original model. Imposing the equality constraint $\alpha = \beta = \theta \in [0, 1]$ yields the symmetric, one-parameter extension
\begin{equation}
\label{eq:indicator:symmetric}
  \ell_\theta(x, y) = \theta \, \ell(x, y) + (1 - \theta) \, (x + y).
\end{equation}


In higher dimensions, a vector of indicators $(I_1, \ldots, I_d)$ can serve to switch some components $X_j$ `on' or `off'. The dependence structure in these indicators then yields an extremal dependence structure for the random vector $X$. Specifically, let $p_j = \Pr[ I_j = 1 ]$; we assume $0 < p_j \le 1$. Let $\bm A = (A_1, \ldots, A_d)$ be a random vector independent of $(I_1, \ldots, I_d)$ and such that $\expec[ \max( A_j, 0 ) ] = 1$ for all $j = 1, \ldots, d$. Then we can define a stable tail dependence function via
\begin{align}
\label{eq:indicator:mv}
  \ell_{\bm p} (x_1, \ldots, x_d)
  &= \expec[ \max (p_1^{-1} x_1 I_1 A_1, \ldots, p_d^{-1} x_d I_d A_d, 0 ) ] \\
  &= \sum_{\varnothing \ne c \subset \{1, \ldots, d\}} p(c) \, \expec[ \max (p_j^{-1} x_j A_j : j \in c) ] \nonumber
\end{align}
where $p(c) = \Pr[ \{ j = 1, \ldots, d : I_j = 1 \} = c ]$. Keeping the law of $\bm A$ fixed, the parameter vector $\bm p$ is equal to a probability distribution $(p(c))_c$ on the non-empty subsets $c$ of $\{1, \ldots, d\}$.

In this way, hierarchical structures can easily be built. For instance, in dimension $d = 4$, we can think of $(I_1, I_2, I_3, I_4) = (J, J, K, K)$. One can also think of logit-type models for the indicators.
\end{example}

\begin{example}[Marshall--Olkin copula]
Applying the device in \eqref{eq:indicator:asymmetric} to the function $\ell(x, y) = \max(x, y)$ yields the model
\[
  \ell_{\alpha, \beta}(x, y) = \max (\alpha x, \beta y) + (1 - \alpha) x + (1 - \beta) y = x + y - \min(\alpha x, \beta y).
\]
The extreme-value copula associated to $\ell_{\alpha, \beta}$ is the Marshall--Olkin copula
\[
  C_{\alpha, \beta}(u, v) = u v \, \min(u^{-\alpha}, v^{-\beta}) = \min(u^{1-\alpha} v, \, u v^{1-\beta}), \qquad (u, v) \in [0, 1]^2.
\]
In higher dimensions, applying the device \eqref{eq:indicator:mv} to the function $\ell( \bm x ) = \max( \bm x )$, that is, $A_j = 1$ with probability one, we get the model
\[
  \ell_{\bm p}(x_1, \ldots, x_d) = \sum_{\varnothing \ne c \subset \{1, \ldots, d\}} p(c) \, \max (p_j^{-1} x_j : j \in c),
\]
the spectral measure of which is discrete. Another stochastic interpretation of this model is provided in \cite{embrechts:lindskog:mcneil:2003}.
\end{example}

\begin{example}[Dirichlet model]
\label{ex:Dirichlet}
Let $\alpha_1, \ldots, \alpha_d$ be positive parameters and let $Z_1, \ldots, Z_d$ be independent Gamma($\alpha_j, 1$) random variables, that is, the density of $Z_j$ is
\[
  f_j(z) = \frac{1}{\Gamma(\alpha_j)} z^{\alpha_j-1} e^{-z}, \qquad 0 < z < \infty.
\]
Put $A_j = \alpha_j^{-1} Z_j$, a positive random variable with unit expectation. The stable tail dependence function associated with the random vector $\bm A = (A_1, \ldots, A_d)$ is
\begin{align*}
  \ell_{\bm A}(x_1, \ldots, x_d)
  &= \expec [ \max( x_1 A_1, \ldots, x_d A_d ) ] \\
  &= \expec [ \max( \alpha_1^{-1} x_1 Z_1, \ldots, \alpha_d^{-1} x_d Z_d ) ] \\
  &= \expec [ (Z_1 + \cdots + Z_d) \, \max( \alpha_1^{-1} x_1 V_1, \ldots, \alpha_d^{-1} x_d V_d ) ],
\end{align*}
where $V_j = Z_j / (Z_1 + \cdots + Z_d)$. It is well known that the random vector $(V_1, \ldots, V_d)$ is independent of $Z_1 + \cdots + Z_d$ and has a Dirichlet distribution with parameters $(\alpha_1, \ldots, \alpha_d)$. We find
\begin{align}
\label{eq:dirichlet}
  \ell_{\bm A}( \bm x ) 
  &= \expec [ (\alpha_1 + \cdots + \alpha_d) \, \max( \alpha_1^{-1} x_1 V_1, \ldots, \alpha_d^{-1} x_d V_d ) ] \\
  &= \frac{\Gamma(\sum_{j=1}^d \alpha_j + 1)}{\prod_{j=1}^d \Gamma(\alpha_j)}
  \int_{\Delta_{d-1}} \max_{j=1, \ldots, d} (\alpha_j^{-1} x_j v_j) \; \prod_{j=1}^d v_j^{\alpha_j-1} \; \diff v_1 \, \cdots \, \diff v_{d-1}. \nonumber
\end{align}
We recognize the Dirichlet model introduced in \cite{coles:tawn:1991}.
\end{example}

\begin{example}[Polynomial Pickands functions]
In the Dirichlet model \eqref{eq:dirichlet}, put $d = 2$ and $\alpha_1 = \alpha_2 = 1$, to obtain
\[
  \ell(x, y) = 2 \int_0^1 \max \{ x v, y (1-v) \} \, \diff v = x + y - \frac{xy}{x+y}.
\]
The Pickands dependence function associated to $\ell$ is $D(t) = \ell(1-t, t) = 1 - t(1-t)$ for $t \in [0, 1]$. Applying the transformation in \eqref{eq:indicator:symmetric} yields the mixture model
\[
  D(t) = 1 - \theta \, t(1-t), \qquad t \in [0, 1],
\]
in terms of a parameter $\theta \in [0, 1]$, introduced in \cite{tawn:1988}. This is the only model in which $D$ is a quadratic polynomial. Applying the transformation in \eqref{eq:indicator:asymmetric} yields the rational model
\begin{align*}
  D_{\alpha, \beta}(t) 
  &= \ell( \alpha (1-t), \beta t ) + (1 - \alpha) (1-t) + (1 - \beta) t \\
  &= 1 - \frac{\alpha \beta t (1-t)}{\alpha(1 - t) + \beta t}, \qquad t \in [0, 1],
\end{align*}
in terms of parameters $\alpha, \beta \in (0, 1]$.

More generally, bivariate max-stable models of which the Pickands dependence function is a higher-order polynomial can be obtained via the transformation in \eqref{eq:indicator:symmetric} applied to the function $\ell(x, y) = \expec[ \max( xA, yB ) ]$ when $A$ and $B$ are random sums
\begin{align*}
  A &= E_1 + \cdots + E_J, & B &= F_1 + \cdots + F_K,
\end{align*}
in terms of independent random variables $J, K, E_1, F_1, E_2, F_2, \ldots$ such that $J$ and $K$ are counting random variables having finite support and unit expectation and where $E_1, F_1, E_2, F_2, \ldots$ are unit exponential random variables. Polynomial Pickands dependence functions are studied in \cite{guillotte:perron:2012}.
\end{example}

\begin{example}[Schlather model]
Let $(S, T)$ be a bivariate normal random vector with standard normal margins and with correlation $\rho \in (-1, 1)$. Put $A = \sqrt{2\pi} S$ and $B = \sqrt{2\pi} T$, so that $\expec[ \max(A, 0) ] = \expec[ \max(B, 0) ] = 1$. The stable tail dependence function corresponding to the random pair $(A, B)$ is
\begin{align*}
  \ell_\rho(x, y) 
  &= \expec[ \max( xA, yB, 0 ) ] \\
  &= \int_{s = -\infty}^0 \int_{t = 0}^\infty \sqrt{2\pi} y t \, \frac{1}{2 \pi} e^{-(s^2 + 2\rho s t + t^2)/2} \, \diff s \, \diff t \\
  &\quad \mbox{} + \int_{s = 0}^\infty \int_{t = -\infty}^0 \sqrt{2\pi} x s \, \frac{1}{2 \pi} e^{-(s^2 + 2\rho s t + t^2)/2} \, \diff s \, \diff t \\
  &\quad \mbox{} + \int_0^\infty \int_0^\infty \sqrt{2\pi} \max(xs, yt) \, \, \frac{1}{2 \pi} e^{-(s^2 + 2\rho s t + t^2)/2} \, \diff s \, \diff t.
\end{align*}
After some tedious calculations, this can be shown to be
\[
  \ell_\rho(x, y) = \frac{1}{2} (x + y) \biggl(1 + \sqrt{1 - 2 (\rho+1) \frac{xy}{(x+y)^2}} \biggr),
\]
see \cite{schlather:2002}. The Pickands dependence function corresponding to this model is
\[
  D_\rho(t) = \frac{1}{2} \bigl\{ 1 + \sqrt{ 1 - 2 (\rho+1) t(1-t) } \bigr\}, \qquad t \in [0, 1].
\]
\end{example}


\begin{example}[Hüsler--Reiss model]
Let $(S, T)$ be a bivariate normal random vector with standard normal margins and with correlation $\rho \in (-1, 1)$. Let $\sigma > 0$ and put
\begin{align*}
  A &= \exp( \sigma S - \sigma^2/2 ), & B &= \exp( \sigma T - \sigma^2/2 ).
\end{align*}
The pair $(A, B)$ is lognormal with unit expectations, and it yields the stable tail dependence function
\begin{align*}
  \ell(x, y)
  = \int_{-\infty}^\infty \int_{-\infty}^\infty \max \{ x \, e^{\sigma s}, \, y \, e^{\sigma t} \} e^{- \sigma^2/2} \, \frac{1}{2\pi} e^{-(s^2 + 2\rho st + t^2)/2} \, \diff s \, \diff t.
\end{align*}
The double integral can be calculated in terms of the standard normal cumulative distribution function $\Phi$, yielding
\[
  \ell_a(x, y) = x \, \Phi \biggl( \frac{a}{2} + \frac{1}{a} \log(x/y) \biggr) + y \, \Phi \biggl( \frac{a}{2} + \frac{1}{a} \log(y/x) \biggr)
\]
with parameter $a = \sigma \sqrt{2(1-\rho)}$. This is the model introduced in \cite{husler:reiss:1989}.

\end{example}

\begin{acknowledgments}
The author wishes to thank the journal Editors for the invitation to write this paper and for a detailed proofreading of an earlier version of the manuscript. The paper originates from a lecture given at the workshop on Copula Models and Dependence at the Centre de Recherches Mathématiques, June 6--9, 2011, in Montréal. Organizers and participants of the workshop are gratefully acknowledged for their support and their encouraging discussions, from which the paper has strongly benefited. 

The author's research was supported by IAP research network grant No.\ P6/03 of the Belgian government (Belgian Science Policy) and by contract No.\ 07/12/002 of the Projet d'Actions de Recherche Concert\'ees of the Communaut\'e fran\c{c}aise de Belgique, granted by the Acad\'emie universitaire Louvain.
\end{acknowledgments}

\bibliographystyle{chicago}
\bibliography{biblio}

\end{document}